\documentclass[9pt,shortpaper,twoside,web]{ieeecolor}

\usepackage{mathdots}

\usepackage[hidelinks=true]{hyperref}
\let\labelindent\relax
\usepackage{enumitem}
\usepackage{generic}
\usepackage{comment}
\usepackage{mathtools,amsmath}
\mathtoolsset{showonlyrefs}
\usepackage{tabularx} 
\usepackage{diagbox}
\usepackage{graphicx}
\usepackage{mathrsfs}
\usepackage[noadjust]{cite}
\usepackage{amssymb,amsfonts}
\usepackage{algorithmic}
\usepackage{tikz}
\usepackage{pgfplots}
\usepackage{pgfplotstable}
\usepgfplotslibrary{units}
\usetikzlibrary{shapes,arrows,patterns,positioning}
\tikzstyle{block} = [draw, fill=gray!20, rectangle, 
    minimum height=2em, minimum width=4em]
\tikzstyle{sum} = [draw, fill=gray!20, circle, node distance=1.5cm]
\tikzstyle{input} = [coordinate]
\tikzstyle{output} = [coordinate]
\tikzstyle{pinstyle} = [pin edge={to-,thin,black}]

\DeclareMathOperator{\rank}{rank}
\DeclareMathOperator{\tr}{tr}

\DeclareMathOperator{\im}{im}

\newcommand{\bbma}{\begin{bmatrix}}
\newcommand{\ebma}{\end{bmatrix}}
\newcommand{\beq}{\begin{equation}}
\newcommand{\eeq}{\end{equation}}
\newcommand{\R}{\mathbb{R}}
\newcommand{\mac}[1]{\mathcal{#1}}

\usepackage{float}
\usepackage{accents}
\usepackage{multicol}
\usepackage{multirow}
\usepackage[ruled,vlined]{algorithm2e}
\usepackage{array,multirow}
\usepackage{hhline}
\usepackage{arydshln}

\newtheorem{theorem}{Theorem}

\newtheorem{lemma}[theorem]{Lemma}
\newtheorem{example}[theorem]{Example}

\newtheorem{proposition}[theorem]{Proposition}

\newtheorem{remark}[theorem]{Remark}

\newtheorem{definition}[theorem]{Definition}

\usepackage{textcomp}
\def\BibTeX{{\rm B\kern-.05em{\sc i\kern-.025em b}\kern-.08em
    T\kern-.1667em\lower.7ex\hbox{E}\kern-.125emX}}
\begin{document}
\title{Fragility Analysis of Data-Driven Feedback Gains}
\author{Yongzhang Li, Amir Shakouri, M. Kanat Camlibel%, \IEEEmembership{Member, IEEE}
%\thanks{This paragraph of the first footnote will contain the date on which you submitted your paper for review. It will also contain support information, including sponsor and financial support acknowledgment. For example, ``This work was supported in part by the U.S. Department of Commerce under Grant 123456.'' }
\thanks{Yongzhang Li and Amir Shakouri contributed equally to this work. The work of Yongzhang Li was supported by the China Scholarship Council under Grant 202106070023.}
\thanks{The authors are with the Bernoulli Institute for Mathematics, Computer Science and Artificial Intelligence, University of Groningen, The Netherlands (e-mail: yongzhang.li@rug.nl, a.shakouri@rug.nl, m.k.camlibel@rug.nl).}
}

\maketitle
\thispagestyle{empty}
\pagestyle{empty}

\begin{abstract}
For linear time-invariant systems, input-state data collected during an open-loop experiment can remedy the lack of knowledge of system parameters. However, such data do not contain information about other system uncertainties such as feedback perturbations. In this paper, we study the effect of additive perturbations on control parameters in a data-based setting.  To this end, we parameterize the set of quadratically stabilizing feedback gains obtained from noisy input-state data. We study the case where a stabilizing data-driven feedback gain is extremely sensitive to feedback perturbations, i.e., a small perturbation in the control parameters, no matter how small, could destabilize the unknown true system. We refer to this case as extreme \emph{fragility} for which we provide a full characterization. We also present necessary and sufficient conditions for the case where the closed-loop system is completely \emph{immune} to feedback perturbations. For the general case where the feedback gain is neither extremely fragile nor immune, we provide a measure by which one can quantize the control fragility directly based on the collected data. We also study the problem of designing the least fragile data-driven feedback gain. The results are presented in terms of linear matrix inequalities and semi-definite programs. 
% either in closed-form, or
\end{abstract}


\section{Introduction}
\label{sec:int}

\IEEEPARstart{D}{ata-driven} control aims at designing a feedback law directly based on data collected from an \emph{unknown system}, bypassing a system identification process \cite{de2019formulas,van2020data,HenkSlemma} (see \cite{hou2013model,van2023informativity,berberich2025overview} and \cite[Ch.~1.2]{DBLSCT2025} for historical backgrounds). Under suitable conditions on the data, such feedback laws can be obtained, and if implemented \emph{exactly} as designed, they stabilize the unknown system. Nevertheless, data collected in an open-loop scenario do not contain information about perturbations arising from the feedback loop. As a result, a data-driven feedback that is unaware of such perturbations might be \emph{fragile}.  In this work, we study data-driven stabilization in the presence of perturbations on the controller parameters and the extent to which the stability of the unknown system is immune to such perturbations. 

From design to real-world implementation, a feedback gain might undergo various perturbations due to, e.g., truncation errors in the numerical computations, limited word-length implementation, conversions from analog to digital and vice versa, and instrumental precision. Some feedback perturbations caused by, e.g., faults in the sensors and actuators, can also be modeled by perturbations on the feedback gain. Moreover, in the data-driven setting, the feedback gain may be subject to additional sources of perturbation. For instance, when the feedback gain is computed directly from measured data, poisoning attacks on the data can induce perturbations in the feedback gain \cite{russo2021poisoning,yu2023poisoning,fotiadis2024poisoning}, potentially degrading performance or even destabilizing the system. In addition, it is appealing for a designed feedback to provide room for future adjustments caused by a change in the design objectives. Hence, it is important to design a feedback that is not extremely sensitive to variations in the control parameters, and to that end, it is useful to know the tolerance towards feedback perturbations. 

The problem of fragile controllers was raised in the paper by Keel and Bhattacharyya \cite{keel1997robust} (see also \cite{makila1998comments,dorato1998non} and the references therein). They showed that a feedback gain obtained through $H_\infty$, $H_2$, $l_1$, or $\mu$ formulations can be fragile, which means that the stability of the closed-loop system is highly sensitive to variations in the control parameters. Since then, several design methods have been proposed to prevent fragile controllers, including the following works: Control fragility caused by fixed word-length implementations is investigated in \cite{paattilammi2000fragility}, where an industrial case study is analyzed, and a loop-shaping method is introduced as a solution. Nonfragile controllers with linear-quadratic performance indices are studied in \cite{famularo2000robust}, where it is observed that for some structured perturbations the problem can be tackled by convex optimization. The use of fixed-structure controllers to avoid fragility is considered in \cite{keel1999robust,haddad2000robust}. Nonfragile design through minimization of pole
sensitivity is proposed in \cite{whidborne2001reduction} and the ellipsoidal sets of nonfragile feedback gains are studied in \cite{peaucelle2005ellipsoidal}. While most of the literature on nonfragile control is devoted to additive perturbations on the feedback gain, multiplicative perturbations are also studied, e.g, in \cite{yang2000guaranteed,yang2001non}. Moreover, nonfragile filter design has also been a topic of research in the literature \cite{yang2001robust,yang2008non}. 

For unknown linear time-invariant (LTI) systems, a stabilizing state-feedback gain can be directly obtained from a collection of noisy input-state data. Given that the noise belongs to a known deterministic model, the data give rise to the set of 
\emph{data-consistent systems}, which includes all systems that could have generated the available data for some noise sequence agreeing with the noise model. In this setting, since any data-consistent system can potentially be the unknown true system, one may seek a feedback gain to stabilize all data-consistent systems. This problem has already been studied within the framework of \emph{data informativity}, for which solutions are presented in, e.g., \cite{van2020data,HenkSlemma}. To the authors' knowledge, however, the fragility issues of such data-driven feedback design methods have not been addressed in the literature. 

In this work, we answer the following question: \emph{Given a state-feedback gain that stabilizes all data-consistent systems, to what extent such a guarantee is intact when the gain is perturbed?} We study the effect of additive perturbations on data-driven feedback gains. To that end, we parameterize all quadratically stabilizing feedback gains, for which we leverage the recently developed tools on quadratic matrix inequalities (QMIs) in \cite{HenkQMI2023,shakouri2024system}. We call a data-driven feedback gain extremely fragile if a variation in the feedback gain, no matter how small in magnitude, destabilizes a subset of the data-consistent systems. Two extreme cases where the data-driven feedback gain is extremely fragile or is completely immune to perturbations are isolated by necessary and sufficient conditions. Next, we study the general case where the data-driven feedback is neither fragile nor immune, for which we characterize the set of perturbations that leave the stability guarantee intact. For this, we introduce a measure that quantifies the fragility of a feedback gain. We show that one can compute this measure and find the least fragile feedback gain by solving a semi-definite program (SDP). For the sake of completeness and better understanding of the introduced concepts, we first study the underlying problems in the model-based setting, and then extend our study to the data-driven framework. 

This note is organized as follows: Section~\ref{sec:II} provides a recap of the data-driven stabilization theory. In Section~\ref{sec:III}, we present a parametrization for the set of data-driven stabilizing feedback gains. In Section~\ref{sec:IV}, we study the fragility of feedback gains within model-based and data-driven settings. Finally, Section~\ref{sec:V} concludes the paper. 

\subsubsection*{Notation}

The set of $n\times n$ real symmetric matrices is denoted by $\mathbb{S}^n$. We say matrix $M$ is positive definite (resp., positive semi-definite), and we denote it by $M>0$ (resp., $M\geq 0$), if $M\in\mathbb{S}^{n}$ and all its eigenvalues are positive (resp., nonnegative). We say matrix $M$ is negative definite (resp., negative semi-definite), and we denote it by $M<0$ (resp., $M\leq 0$), if $-M>0$ (resp., \mbox{$-M\geq 0$}). We denote the spectral norm of $M \in \mathbb{C}^{p \times q}$ by $\|M\|$. Given \mbox{$C\in\mathbb{R}^{p\times q}$} and $\rho\geq 0$, let $\mathcal{B}(C,\rho)\coloneqq\{X:\|X-C\|\leq\rho\}$ denote the ball centered around $C$ with radius $\rho$. For a matrix $M\in\mathbb{R}^{p\times q}$, its Moore–Penrose pseudoinverse is denoted by $M^\dagger$. We define the set 
\begin{equation}
\begin{split}
\pmb{\Pi}_{q,r}\coloneqq \bigg\{\begin{bmatrix}
\Pi_{11} & \Pi_{12} \\
\Pi_{21} & \Pi_{22}
\end{bmatrix}\in\mathbb{S}^{q+r}: \Pi_{11}\in\mathbb{S}^q,\Pi_{22}\in\mathbb{S}^r,\\
\Pi_{22}\leq 0,\Pi|\Pi_{22}\geq 0,\ker\Pi_{22}\subseteq\ker\Pi_{12}\bigg\},
\end{split}
\end{equation}
where $\Pi|\Pi_{22}\coloneqq \Pi_{11}-\Pi_{12} \Pi_{22}^{\dagger} \Pi_{21}$ is the generalized Schur complement of $\Pi$ with respect to $\Pi_{22}$. We also define the QMI-induced sets $\mathcal{Z}_r(\Pi)$ and $\mathcal{Z}^+_r(\Pi)$ as follows:
\begin{equation}
\label{eq:Zp}
\begin{split}
\mathcal{Z}_r(\Pi)\coloneqq \left\{Z\in\mathbb{R}^{r\times q}:\begin{bmatrix}
I \\ Z
\end{bmatrix}^\top \Pi \begin{bmatrix}
I \\ Z
\end{bmatrix}\geq 0\right\},  \\
\mathcal{Z}^+_r(\Pi)\coloneqq \left\{Z\in\mathbb{R}^{r\times q}:\begin{bmatrix}
I \\ Z
\end{bmatrix}^\top \Pi \begin{bmatrix}
I \\ Z
\end{bmatrix} > 0\right\}.
\end{split}
\end{equation}

\section{Recap of Data-driven Stabilization}
\label{sec:II}
In this section, we briefly review data-driven stabilization results within the \emph{data informativity} framework \cite{van2020data,HenkSlemma,van2023informativity,DBLSCT2025} based on the existing literature. %\cite{HenkSlemma,HenkQMI2023}. 

Consider an LTI system of the form
\beq
  x(t+1)=A_\text{true}x(t)+B_\text{true}u(t)+w(t),
  \label{dy-lsys}
\eeq
where $x(t)\in\R^n$ is the state, $u(t) \in \R^m$ is the input, and $w(t)\in\R^n$ is the process noise. We assume that $A_\text{true}$ and $B_\text{true}$ are unknown, but we can access input-state data
\begin{equation}
\label{eq:data}
\mathcal{D}\!\coloneqq\!\left(\begin{bmatrix}
    u(0) \!\!&\!\! u(1)  \!\!&\!\! \cdots \!\!&\!\! u(T-1)
\end{bmatrix}, \begin{bmatrix}
    x(0) \!\!&\!\! x(1) \!\!&\!\! \cdots \!\!&\!\! x(T)
\end{bmatrix}\right)
\end{equation} 
collected from system \eqref{dy-lsys}. These data are influenced by the unknown process noise $W_-\coloneqq\begin{bmatrix}
w(0) & w(1) & \cdots & w(T-1)
\end{bmatrix}$. We assume that the noise signal satisfies an energy bound of the form
\begin{equation}
\label{eq:ass1-1}
W_-^\top\in\mathcal{Z}_T(\Phi),
\end{equation} 
where $\Phi\in\pmb{\Pi}_{n,T}$ with $\Phi_{22}<0$ is given.

Several noise models can be captured by this energy bound (see \cite[p. 4]{HenkQMI2023} for a detailed account), among which the noise-free case corresponds to $\Phi_{11}=0$, $\Phi_{12}=\Phi_{21}^\top=0$, and $\Phi_{22}=-I$. 

We define matrices $X_-\coloneqq\bbma x(0) & x(1) & \cdots & x(T-1) \ebma$ and $X_+\coloneqq\bbma x(1) & x(2) & \cdots & x(T) \ebma$. We also define the matrix $U_-\coloneqq\bbma u(0) & u(1) & \cdots & u(T-1) \ebma$.
%\begin{equation}
%\begin{split}
%U_-&\coloneqq\bbma u(0) & u(1) & \cdots & u(T-1) \ebma,  \\
%X_-&\coloneqq\bbma x(0) & x(1) & \cdots & x(T-1) \ebma, \\
%X_+&\coloneqq\bbma x(1) & x(2) & \cdots & x(T) \ebma.
%\end{split}
%\end{equation}
A pair of real matrices $(A,B)$ is called a \emph{data-consistent system} if it satisfies
\begin{equation}
X_+=AX_-+BU_-+W_-
\end{equation}
for some $W_-$ satisfying \eqref{eq:ass1-1}. We define the set of all data-consistent systems as
\begin{equation}
\label{eq:8}
\Sigma_\mathcal{D}\coloneqq \{(A,B):(X_+-AX_--BU_-)^\top\in \mathcal{Z}_{T}(\Phi)\}.
\end{equation}
Clearly, $(A_\textup{true},B_\textup{true}) \in \Sigma_\mathcal{D}$. One can verify that (see \cite[Lem. 4]{HenkSlemma})
\begin{equation}
\label{eq:explaining_QMI_induced}
\Sigma_\mathcal{D}=\left\{(A,B): \begin{bmatrix}
A & B
\end{bmatrix}^\top\in\mathcal{Z}_{n+m}(N)\right\},
\end{equation}
where
\begin{equation}
N\coloneqq\bbma N_{11} & N_{12} & N_{13} \\ N_{21} & N_{22} & N_{23} \\ N_{31} & N_{32} & N_{33} \ebma = \bbma I & X_+ \\ 0 & -X_- \\ 0 & -U_- \ebma \Phi \bbma I & X_+ \\ 0 & -X_- \\ 0 & -U_- \ebma^\top,
\label{eq:def_N}
\end{equation} 
$N_{11},N_{22} \in \mathbb{S}^n$, and $N_{33} \in \mathbb{S}^m$. We also define
\begin{equation}
\bar{N}_{22}=\begin{bmatrix}
N_{22} & N_{23} \\ N_{32} & N_{33}
\end{bmatrix}\text{ and }\bar{N}_{21}=\bar{N}_{12}^\top=\begin{bmatrix}
N_{21} \\ N_{31} 
\end{bmatrix}.
\end{equation}
The following proposition summarizes some facts about $\Sigma_\mathcal{D}$. 
\begin{proposition}[{\cite[Prop. 14]{shakouri2024system}}]
\label{prop:Sigma_properties}
%Let the data $\mathcal{D}$ be given. 
Let $\Phi\in\pmb{\Pi}_{n,T}$ with $\Phi_{22}<0$. Then, the following statements hold:
\begin{enumerate}[label=(\alph*),ref=\ref{prop:Sigma_properties}(\alph*)]
    \item\label{prop:Sigma_properties(b)} $\Sigma_\mathcal{D}$ is bounded \emph{if and only if} $\rank\begin{bmatrix}
X_-^\top & U_-^\top
\end{bmatrix}=n+m$.    \item\label{prop:Sigma_properties(c)} Assume that $\Sigma_\mathcal{D}$ is bounded. Then, $\Sigma_\mathcal{D}$ is a singleton \emph{if and only if} $N|\bar{N}_{22}=0$.
%\begin{equation}
%    \label{eq:N_schur}
%    N|\bar{N}_{22}=0.
%    \end{equation}
In this case, $\Sigma_\mathcal{D}=\{(A_{\text{true}},B_{\text{true}})\}$ and
    \begin{equation}
    \begin{bmatrix} A_{\text{true}} & B_{\text{true}} \end{bmatrix}=-\bar{N}_{12}\bar{N}_{22}^{-1}.
    \end{equation}
\end{enumerate}
\end{proposition}

We consider the following notions of data informativity. 
\begin{definition}[\cite{DBLSCT2025}]
\label{def:1}
The data $\mathcal{D}$ are called 
\begin{enumerate}[label=(\alph*),ref=\ref{def:1}(\alph*)]
    \item\label{def:1(a)} \emph{informative for stabilization} if there exists $K\in\mathbb{R}^{m\times n}$ such that $A+BK$ is Schur for all $(A,B)\in\Sigma_\mathcal{D}$.
    \item\label{def:1(b)} \emph{informative for quadratic stabilization} if there exist \mbox{$K\in\mathbb{R}^{m\times n}$} and $P>0$ such that $P-(A+BK) P(A+BK)^\top>0$ for all $(A,B)\in\Sigma_\mathcal{D}$.
\end{enumerate}
\end{definition}

Under suitable conditions, the two notions of data informativity in Definition~\ref{def:1} are equivalent \cite{van2022data}. One such condition corresponds to the noise-free case.

\begin{proposition}[{\cite[Thm. 1]{van2022data}}]
Suppose that $\Phi_{11}=0$, $\Phi_{22}=-I$, and $\Phi_{12}=\Phi_{21}^\top=0$. Then, the data $\mathcal{D}$ are informative for stabilization \emph{if and only if} they are so for quadratic stabilization. 
\end{proposition}

The following proposition provides a necessary and sufficient LMI condition for the informativity of data $\mac{D}$ for quadratic stabilization. 
%This LMI is the nonstrict version of \cite[Thm. 5.1(b)]{HenkQMI2023}. The proof of this proposition is omitted here as it follows similar lines to the proof of \cite[Thm. 5.1(b)]{HenkQMI2023} by using the nonstrict Lyapunov inequality instead of the strict version (see Lemma \ref{lem:newlyp}).

\begin{proposition}[{\cite[Thm. 5.1(b)]{HenkQMI2023}}]
\label{prop:DDC}
 Suppose that $\bbma X_-^\top & U_-^\top \ebma$ has full column rank. Then, the data $\mac{D}$ are informative for quadratic stabilization \emph{if and only if} there exist $P>0$, $\alpha\geq 0$, and $L$ such that
\begin{equation}
\bbma P & 0 & 0 & 0 \\ 0 & -P & - L^\top & 0 \\ 0 & - L & 0 &  L \\ 0 & 0 &  L^\top & P \ebma
-\alpha \bbma N & 0\\ 0 & 0 \ebma > 0.
\label{st-ineq}
\end{equation}
Moreover, $K=LP^{-1}$ is a stabilizing feedback gain for all \mbox{$(A,B) \in \Sigma_\mathcal{D}$} provided that \eqref{st-ineq} is feasible.
\end{proposition}

%In particular, for the noiseless case, the feasibility of LMI \eqref{nst-ineq} turns out to be a necessary and sufficient condition for informativity for stabilization (see \cite[Thm. 2]{van2021matrix}), where the conservatism of stabilizing with a common Lyapunov function is lifted. 

%We note that the LMI \eqref{nst-ineq} can be feasible even if $\begin{bmatrix} X_-^\top & U_-^\top \end{bmatrix}$ is rank deficient (see \cite[Ex. 19]{van2020data}). This means that the data can potentially be informative for (quadratic) stabilization in the case that the set of explaining systems $\Sigma_\mathcal{D}$ is unbounded (see Proposition \ref{prop:Sigma_properties(b)}). %We also note that $\rank X_-=n$ is a necessary condition for the data to be informative for (quadratic) stabilization (see \cite[Thm. 16]{van2020data}.

 The LMI \eqref{st-ineq} contains $\frac{n(n+1)}{2}+mn+1$ unknowns. The computational complexity of this LMI can be reduced. To this end, we define
\begin{equation}
\begin{split}
\Gamma &\coloneqq \begin{bmatrix}
P & 0 & 0 \\
0 & 0 & 0 \\
0 & 0 & 0
\end{bmatrix}-\alpha N,\\
\Theta&\coloneqq\begin{bmatrix}
P & 0 \\
0 & -P
\end{bmatrix}-\alpha\begin{bmatrix}
N_{11} & N_{12} \\
N_{21} & N_{22}
\end{bmatrix}, \\
M&\coloneqq\begin{bmatrix}
-\alpha N_{33} & 0 \\
0 & -P
\end{bmatrix}-\begin{bmatrix}
\alpha N_{31} & \alpha N_{32} \\
0 & P
\end{bmatrix}\Theta^{\dagger} \begin{bmatrix}
\alpha N_{13} & 0 \\
\alpha N_{23} & P
\end{bmatrix}\!\!.
\end{split}
\end{equation}	
We partition matrix $M$ into four blocks, where the first block is denoted by $M_{11} \in \R^{m \times m}$, the second and third blocks are denoted by $M_{12}=M_{21}^\top \in \R^{m \times n}$, and the last block is denoted by $M_{22} \in \R^{n \times n}$. Now, the following proposition provides alternative LMI conditions to those in Proposition~\ref{prop:DDC}, reducing the number of unknowns to $\frac{n(n+1)}{2}+1$. %The proof of this result is omitted here as it follows the footsteps of the proof of \cite[Thm. 5.3(b)]{HenkQMI2023} by using the nonstrict Lyapunov inequality instead of the strict version (see Lemma \ref{lem:newlyp}).

\begin{proposition}[{\cite[Thm. 5.3(b)]{HenkQMI2023}}]
\label{prop:1p}
Suppose that $\begin{bmatrix}
X_-^\top & U_-^\top
\end{bmatrix}$ has full column rank. Then, the data $\mac{D}$ are informative for quadratic stabilization \emph{if and only if} there exist $P > 0$ and $\alpha\geq 0$ such that 
\begin{equation}
\label{eq:lmi1}
\Gamma > 0 \ \text{ and } \ \Theta > 0.
\end{equation}
Moreover, $K=-M_{12}M_{22}^{-1}$ is a stabilizing feedback gain for all $(A,B) \in \Sigma_\mathcal{D}$ provided that \eqref{eq:lmi1} is feasible. 
\end{proposition}

Propositions~\ref{prop:DDC} and~\ref{prop:1p} can be used to obtain a stabilizing feedback gain for the unknown true system based on the collected data. However, the extent of the freedom over the choice of the feedback gain is not provided by these results. Such a description, which is instrumental in the fragility analysis of data-driven feedback gains, can be studied by the parametrization provided in the next section.  

\section{Parametrization of Data-driven Feedback Gains}
\label{sec:III}

In this section, we will parameterize the set of quadratically stabilizing feedback gains that can be obtained from the collected input-state data. To that end, suppose that 
$\begin{bmatrix}
X_-^\top & U_-^\top
\end{bmatrix}$ has full column rank. For $P > 0$ and $\alpha\geq 0$, we define
\begin{equation}
\begin{split}
\mathcal{K}(P,\alpha)\coloneqq \{LP^{-1}:L\ \text{satisfies LMI \eqref{st-ineq} with}\ P\ \text{and}\ \alpha\}.
\end{split}
\end{equation} 
%Based on Proposition \ref{}, $\mathcal{K}(P,\alpha)$ is nonempty for some $P\geq I$ and $\alpha \geq 0$ \emph{if and only if} the data $\mac{D}$ are informative for quadratic stabilization. 
The set of all quadratically stabilizing feedback gains that can be obtained from data $\mathcal{D}$ is the union of the sets $\mathcal{K}(P,\alpha)$ over all $P>0$ and $\alpha \geq 0$, which is denoted by $\mathcal{K}\coloneqq\bigcup_{P>0,\alpha \geq 0} \mathcal{K}(P,\alpha)$. The following theorem provides a parametrization for such sets.
\begin{theorem}
\label{th:1}
Suppose that $\begin{bmatrix}
X_-^\top & U_-^\top
\end{bmatrix}$ has full column rank and data $\mac{D}$ are informative for quadratic stabilization. Let $P > 0$ and $\alpha\geq 0$ satisfy \eqref{eq:lmi1}. Then, $K \in \mac{K}(P,\alpha)$ \emph{if and only if} $K^\top \in \mac{Z}^+_n(M)$. Moreover, we have
\begin{equation}
%\begin{split}
\mathcal{K}(P,\alpha)\!\!=\!\!\{-M_{12}M_{22}^{-1}+(M|M_{22})^{\frac{1}{2}}S(-M_{22})^{-\frac{1}{2}}\!:\! S^\top\!\! S\!<\! I\}.
%\end{split}
\label{set:kp}
\end{equation} 
\end{theorem}
\begin{proof}
First, we prove that $K \in \mac{K}(P,\alpha)$ if and only if \mbox{$K^\top \in \mac{Z}^+_n(M)$}. For the ``only if'' part, let $K \in \mac{K}(P,\alpha)$. This implies that the LMI \eqref{st-ineq} is feasible with $L=KP$. Take the Schur complement of the matrix in \eqref{st-ineq} with respect to its lower block to have
\begin{equation}
\bbma 
P & 0 & 0 \\ 
0 & -P & - L^\top \\ 
0 & - L & -LP^{-1}L^\top 
\ebma -\alpha N > 0.
\label{st-ineq-k2}
\end{equation}
Now, since we have $\Theta>0$, take the Schur complement of \eqref{st-ineq-k2} with respect to the first $2n\times 2n$ block to have
\begin{equation}
\label{eq:schur_comp2}
-LP^{-1}L^\top - \alpha N_{33}-\begin{bmatrix}
\alpha N_{13} \\ L^\top\!\!+\alpha N_{23}
\end{bmatrix}^{\!\top}\!\!\Theta^{-1}\! \begin{bmatrix}
\alpha N_{13} \\ L^\top\!\!+\alpha N_{23}
\end{bmatrix}\!>\!0.
\end{equation}
Using $L=KP$ we observe that \eqref{eq:schur_comp2} can be written as
\begin{equation}
\label{eq:pf_qmi_M}
\begin{bmatrix}
I & K
\end{bmatrix} M\begin{bmatrix}
I & K
\end{bmatrix}^\top> 0.
\end{equation}
Therefore, $K^\top\in\mathcal{Z}^+_n(M)$. To prove the ``if'' part, let \mbox{$K^\top\in\mathcal{Z}^+_n(M)$}, i.e., \eqref{eq:pf_qmi_M} holds. Take $L=KP$. Observe that \eqref{eq:pf_qmi_M} implies \eqref{eq:schur_comp2}. Using a Schur complement argument, one can see that \eqref{st-ineq-k2} holds, and thus the LMI \eqref{st-ineq} is satisfied. Therefore, $LP^{-1}$, and thus $K$, belongs to $\mathcal{K}(P,\alpha)$. 

Next, we use \cite[Thm. 3.3]{HenkQMI2023} to parameterize the set $\mac{Z}^+_n(M)$. Since $P > 0$ and \mbox{$M_{22}\leq -P$}, we have \mbox{$\ker M_{22}=\{0\}$}. This readily implies that \mbox{$\ker M_{22} \subseteq \ker M_{12}$}. In addition, by hypothesis, the inequality \eqref{eq:pf_qmi_M} is feasible, which implies that $\mac{Z}_n (M) \neq \varnothing$. Thus, it follows from \cite[pp. 6,7]{HenkQMI2023} that \mbox{$M|M_{22}\geq 0$}. Therefore, we have $M \in \pmb{\Pi}_{m,n}$. Now, in view of \cite[Thm. 3.3]{HenkQMI2023}, every $K=LP^{-1}$ satisfies \mbox{$K=-M_{12}M_{22}^{-1}+(M|M_{22})^{\frac{1}{2}}S(-M_{22})^{-\frac{1}{2}}$} for some $S$ with $S^\top S < I$.
\end{proof}

\section{Fragility Analysis}
\label{sec:IV}

In this section, we study the effect of \emph{additive} perturbations on the stabilizing feedback gains (see Fig.~\ref{fig:blkdiag_add}) within both model-based and data-driven settings. 

\begin{comment}

First, we consider the model-based setting, and we provide solutions for the following problem:
\begin{center}
\textit{Suppose that $(A,B)$ is stabilizable. Find $K$ and $\rho>0$ such that $A+B(K+\Delta)$ is Schur for all $\Delta\in\mathcal{B}(0,\rho)$.}
\end{center}

Later, we extend our study to the data-driven setting, where we provide solutions to the following problem:

\begin{center}
\textit{Suppose that the data $\mathcal{D}$ are informative for quadratic stabilization. Find $K$ and $\rho>0$ such that $A+B(K+\Delta)$ is Schur for all $(A,B)\in\Sigma_\mathcal{D}$ and all $\Delta\in\mathcal{B}(0,\rho)$.}
\end{center}
    
\end{comment}

\begin{figure}[h!]
    \centering
    \begin{tikzpicture}[
        scale=0.9,
        transform shape,
        auto,
        node distance=2cm,
        >=latex'
    ]
    \node [rounded corners=0.75mm, block ,name=system]{System};
    \node [rounded corners=0.75mm, block, above=0.25cm of system] (D) {$\Delta$};
    \node [sum, left=0.75cm of system] (wnoise) {};
    \node [rounded corners=0.75mm, block, left=0.75cm of wnoise] (K) {$K$};
    \node [input, left=0.75cm of K] (input) {};
    \node [output, right=0.75cm of system] (output) {};
    \node [input, right=0.75cm of D] (dinput) {};
    \node [output, above=0.75cm of wnoise] (doutput) {};
    \node [output, below=0.5cm of system] (feed) {};
    \node [input, right=1.39cm of feed] (finput) {};
    \node [output, below=0.825cm of input] (foutput) {};
    \draw [->] (input) -- (K);
    \draw [->] (K) -- (wnoise);
    \draw [->] (wnoise) -- node {$u$} (system);
    \draw [-] (output) -- (dinput);
    \draw [->] (dinput) -- (D);
    \draw [-] (D) -- (doutput);
    \draw [->] (doutput) -- (wnoise);
    \draw [-] (output) -- (finput);
    \draw [-] (finput) -- (foutput);
    \draw [-] (foutput) -- (input);
    \draw [-] (system) -- node {$x$}(output);
    \end{tikzpicture}
    \caption{Additive feedback perturbations.}
    \label{fig:blkdiag_add}
\end{figure}

%\begin{figure}[h!]
%    \centering
%    \begin{tikzpicture}[auto, node distance=2cm,>=latex']
%    \node [rounded corners=0.75mm, block ,name=system]{System};
%    \node [rounded corners=0.75mm, block, above=0.25cm of system] (D) {$\Delta$};
%    \node [sum, left=0.75cm of system] (wnoise) {};
%    \node [rounded corners=0.75mm, block, left=0.75cm of wnoise] (K) {$K$};
%    \node [input, left=0.75cm of K] (input) {};
%    \node [output, right=0.75cm of system] (output) {};
%    \node [input, right=0.75cm of D] (dinput) {};
%    \node [output, above=0.75cm of wnoise] (doutput) {};
%    \node [output, below=0.5cm of system] (feed) {};
%    \node [input, right=1.39cm of feed] (finput) {};
%    \node [output, below=0.825cm of input] (foutput) {};
%    \draw [->] (input) -- node {} (K);
%    \draw [->] (K) -- node {} (wnoise);
%    \draw [->] (wnoise) -- node {$u$} (system);
%    \draw [-] (output) -- node {}(dinput);
%    \draw [->] (dinput) -- node {}(D);
%    \draw [->] (dinput) -- node {}(D);
%    \draw [-] (D) -- node {}(doutput);
%    \draw [->] (doutput) -- node {}(wnoise);
%    \draw [-] (output) -- node {}(finput);
%    \draw [-] (finput) -- node {}(foutput);
%    \draw [-] (foutput) -- node {}(input);
%    \draw [-] (system) -- node {$x$}(output);
%    \end{tikzpicture}
%    \caption{Additive feedback perturbations.}
%    \label{fig:blkdiag_add}
%\end{figure}

\subsection{Model-based analysis}
\label{sub:frag_mod}

For a stabilizable $(A,B)$, let $K$ be such that $A_K\coloneqq A+BK$ is Schur. We define the radius of stabilizing gains centered around $K$ as
\begin{equation}
\label{eq:psiK}
\mu_{(A,B)}^{\text{a}}(K)\!\!\coloneqq\! \sup \{\rho\!:\!A_K+B\Delta\ \text{is Schur for all}\ \Delta\!\in\!\mathcal{B}(0,\rho)\}.
\end{equation}
The value of $\mu_{(A,B)}^{\text{a}}(K)$ gives the largest bound for an additive feedback perturbation so that the closed-loop system remains stable. The following theorem characterizes the case where $\mu_{(A,B)}^{\text{a}}(K)$ is finite. 

\begin{theorem}
\label{prop:Bnonzero}
Suppose that $(A,B)$ is stabilizable. Let $K$ be such that $A_K$ is Schur. Then, $\mu_{(A,B)}^{\text{a}}(K)>0$. Moreover, $\mu_{(A,B)}^{\text{a}}(K)$ is finite \emph{if and only if} $B\neq 0$. 
\end{theorem}
\begin{proof}
Since $A_K$ is Schur, there exist $P>0$ and \mbox{$\beta>0$} such that the Lyapunov inequality \mbox{$P-A_K PA_K^\top \geq\beta I$} holds. Thus, there exists a sufficiently small $\rho>0$ such that \mbox{$P-(A_K+B\Delta) P(A_K+B\Delta)^\top>0$} holds for all \mbox{$\Delta\in\mathcal{B}(0,\rho)$}. This implies that $\mu_{(A,B)}^{\text{a}}(K)>0$. For the rest, first, suppose that $B=0$. In this case, $\mu_{(A,B)}^{\text{a}}(K)$ is obviously not finite. Next, suppose that $B\neq 0$. Let $\rho>0$ be such that $A_K+B\Delta$ is Schur for all $\Delta\in\mathcal{B}(0,\rho)$. Take $\Delta=(\rho /\|B\|)B^\top$. Since $A_K+B\Delta$ is Schur, we have $|\tr(A_K+B\Delta)|< n$, thus,
$\left|\tr(A_K)+\rho\tr(BB^\top)/\|B\|\right|< n$. This shows that $\rho< (n-\tr(A_K))\|B\|/\tr(BB^\top)$, which proves that $\mu_{(A,B)}^{\text{a}}(K)$ is finite. 
\end{proof}

For SISO systems, the value of $\mu_{(A,B)}^{\text{a}}(K)$ may be obtained using discrete-time counterparts of the Kharitonov's theorem \cite{hollot1986some}. However, for MIMO systems, computing $\mu_{(A,B)}^{\text{a}}(K)$ for given $(A,B)$ and $K$ is a hard problem in general (see, e.g., \cite{nemirovskii1993several}). Therefore, we aim at developing numerically tractable algorithms to approximate $\mu_{(A,B)}^{\text{a}}(K)$ by computing a positive lower bound for it. To that end, given $K$ and $P>0$ such that the Lyapunov inequality $P-A_KPA_K^\top>0$ holds, we define 
\begin{equation}
\label{eq:rhoKP}
\begin{split}
\kappa_{(A,B)}^{\text{a}}(P,K)\coloneqq \sup \{\rho:\text{for every}\ \Delta\in\mathcal{B}(0,\rho)\ \text{we have}\\
P-(A_K+B\Delta) P (A_K+B\Delta)^\top>0\}.
\end{split}
\end{equation}
We also define the largest value of $\kappa_{(A,B)}^{\text{a}}(P,K)$ over $P>0$ as
\begin{equation}
\label{eq:rhoK}
\begin{split}
\lambda_{(A,B)}^{\text{a}}(K)\coloneqq \sup\{\kappa_{(A,B)}^{\text{a}}(P,K):P>0\}.
\end{split}
\end{equation}
It is evident that $\kappa_{(A,B)}^{\text{a}}(P,K)\leq\lambda_{(A,B)}^{\text{a}}(K)\leq\mu_{(A,B)}^{\text{a}}(K)$. Now, we define the largest value of $\lambda_{(A,B)}^{\text{a}}(K)$ over $K$ as
\begin{equation}
\label{eq:def_lam_mod}
\lambda_{(A,B)}^{\text{a}}\coloneqq \sup\{\lambda_{(A,B)}^{\text{a}}(K):K\in\mathbb{R}^{m\times n}\}.
\end{equation}
For a given $(A,B)$, the value of $\lambda_{(A,B)}^{\text{a}}$ can be computed by solving an SDP provided by the following theorem.

%\begin{remark}
%What about uncontrollable but stabilizable systems? 
%\end{remark}
\begin{comment}
\begin{remark}
The result of Theorem \ref{th:frag_model} can also be considered as robustness analysis to uncertainties of \emph{matched} type on the system parameters (see \cite{gutman1979uncertain,corless1981continuous,barmish1985necessary}). In this view, the feedback obtained by Theorem \ref{th:frag-model-opt} provides the maximum radius of quadratic stability against uncertainties of matched type. 
\end{remark}
\end{comment}

\begin{theorem}
\label{th:frag-model-opt}
Suppose that $(A,B)$ is stabilizable and \mbox{$B\neq 0$}. Then, $\lambda_{(A,B)}^{\text{a}}=\sqrt{1/\beta_*}$ where $\beta_*$ is the optimal value of the following SDP: 
\begin{subequations}
\mathtoolsset{showonlyrefs=false}
\label{eq:th:frag-model-opt_s}
\begin{align}
\label{eq:th:frag-model-opt_s(a)}
\min_{Q,L}\ \beta & \\
\label{eq:th:frag-model-opt_s(b)}
\text{s.t.}&\hspace{0.15 cm}
\begin{bmatrix}
Q & QA^\top+L^\top B^\top  & Q \\
AQ+BL & Q-BB^\top  & 0 \\
Q & 0  & \beta I
\end{bmatrix}\geq 0.
\end{align}
\mathtoolsset{showonlyrefs=true}
\end{subequations}
Moreover, provided that $L_*$ and $Q_*$ are optimizers of the SDP~\eqref{eq:th:frag-model-opt_s}, $A+B(K_*+\Delta)$ with $K_*=L_*Q_*^{\dagger}$ is Schur for all $\|\Delta\|<\lambda_{(A,B)}^{\text{a}}$.   
\end{theorem}

To prove this theorem, we need an auxiliary result presented in the following lemma.

\begin{lemma}
\label{lem:mod_slem}
Suppose that $(A,B)$ is stabilizable. Let $K$ and $P>0$ be such that $P-A_KPA_K^\top>0$. Then, we have $\rho<\kappa_{(A,B)}^{\text{a}}(P,K)$ \emph{if and only if} there exists $\gamma> 0$ such that the following LMI is satisfied:
\begin{equation}
\label{eq:addp_mod_opt_com1_s}
    \begin{bmatrix}
         P & 0 \\ 0 & \gamma I
    \end{bmatrix}-\begin{bmatrix}
        A_K & B \\ I & 0 
    \end{bmatrix}^\top \begin{bmatrix}
         P & 0 \\ 0 & \gamma \rho^2 I 
    \end{bmatrix}\begin{bmatrix}
        A_K & B \\ I & 0
    \end{bmatrix}> 0.
\end{equation}
\end{lemma}
\begin{proof}
We define
\begin{equation}
\label{eq:Mrho_s}
%\label{eq:ML_s}
M_\rho \coloneqq \begin{bmatrix}
\rho^2 I_n & 0 \\
0 & -I_m
\end{bmatrix} \ \text{ and } \
M_L= \begin{bmatrix}
    P & 0 \\
    0 & 0
    \end{bmatrix}\!-\!\begin{bmatrix}
    A_K^\top \\
    B^\top
    \end{bmatrix}\! P \!\begin{bmatrix}
    A_K^\top \\
    B^\top
    \end{bmatrix}^\top.
\end{equation}
Observe that $\Delta$ satisfies $P-(A_K+B\Delta)^\top P (A_K+B\Delta) > 0$ for some $P>0$ if and only if $\Delta\in\mathcal{Z}^+_{m}(M_{L})$. We also observe that $\mathcal{B}(0,\rho)=\mathcal{Z}_{m}(M_\rho)$. One can verify that $M_\rho \in \pmb{\Pi}_{n,m}$ and the second diagonal block of $M_\rho$ is negative definite. Therefore, it follows from the matrix S-lemma \cite[Thm. 4.10]{HenkQMI2023} that $\mathcal{Z}_{m}(M_\rho)\subseteq\mathcal{Z}_{m}^+(M_L)$ if and only if  $M_L-\gamma M_\rho> 0$ for some $\gamma\geq 0$. Since the latter is not feasible for $\gamma=0$, it is equivalent to \eqref{eq:addp_mod_opt_com1_s} with $\gamma>0$. 
\end{proof}

The proof of Theorem~\ref{th:frag-model-opt} now follows from Lemma~\ref{lem:mod_slem}.

\textit{Proof of Theorem~\ref{th:frag-model-opt}:}
 We first show that LMI \eqref{eq:th:frag-model-opt_s(b)} is feasible. For this, let $\nu>0$, $K$ and $P>0$ be such that \mbox{$P-(A+BK)P(A+BK)^\top > \nu I$}. Take $Q=\frac{\|BB^\top\|}{\nu}P$ and \mbox{$L=KQ$}. Observe that 
\begin{equation}
\begin{split}
&Q-(AQ+BL)Q^{-1}(AQ+BL)^\top \\
&=Q-(A+BK)Q(A+BK)^\top\geq \|BB^\top\|I \geq BB^\top.
\end{split}
\end{equation} 
This implies that 
\begin{equation}
    \begin{bmatrix}
        Q & QA^\top+L^\top B^\top \\
AQ+BL & Q-BB^\top
    \end{bmatrix} > 0.
\end{equation} Thus, there exists a sufficiently large $\bar\beta \geq 0$ such that \eqref{eq:th:frag-model-opt_s(b)} is feasible for all $\beta\geq\bar \beta$. Now, we show that the SDP \eqref{eq:th:frag-model-opt_s} is attained. It follows from \eqref{eq:th:frag-model-opt_s} that any feasible solution satisfies $\beta \geq 0$. Let $\beta_*$ be the infimum value of $\beta$ such that \eqref{eq:th:frag-model-opt_s(b)} holds for some $Q$ and $L$. We have $\beta_*\leq\bar{\beta}$. Let $(\beta_i,Q_i,L_i)$ be a sequence satisfying \eqref{eq:th:frag-model-opt_s} such that $\lim_{i \rightarrow \infty} \beta_i=\beta_*$. Note that this sequence satisfies $Q_i\leq \beta_i I$. Hence, we have $\lim_{i \rightarrow \infty} Q_i\leq\beta_* I$. One can also verify that there exists $\eta>0$ such that $\lim_{i \rightarrow \infty} \|BL_i\|\leq \eta$. Therefore, $\beta_i$, $Q_i$, and $BL_i$ are all bounded sequences. Let $\bar{L}_i=B^\top (BB^\top)^\dagger BL_i$. Since $BL_i$ is a bounded sequence, $\bar{L}_i$ is a bounded sequence with the property that $BL_i=B\bar{L}_i$. Thus, $(\beta_i,Q_i,\bar{L}_i)$ is a bounded sequence that satisfies $\lim_{i \rightarrow \infty} \beta_i=\beta_*$. Now, it follows from the \mbox{Bolzano-Weierstrass} theorem that the optimal value is attained on a convergent subsequence.  
%theorem that Take a convergent sequence $\bar L_i$ such that $B\bar L_i=BL_i$. Now, $(\beta_i,\bar L_i,Q_i)$ is a convergent sequence with the property that $\lim_{i \rightarrow \infty} \beta_i=\beta_*$. Therefore, the SDP \eqref{eq:th:frag-model-opt_s} is well defined.

Next, we prove that $\lambda_{(A,B)}^{\text{a}}=\sqrt{1/\beta_*}$. For this, we denote
\begin{equation}
\label{eq:mod_slmi}
\mathcal{L}(Q,L,\beta) \coloneqq\begin{bmatrix}
Q & QA^\top+L^\top B^\top  & Q \\
AQ+BL & Q-BB^\top  & 0 \\
Q & 0  & \beta I
\end{bmatrix}.
\end{equation} 
We claim that
\begin{equation}
\label{eq:claim_mod}
    \lambda^\textup{a}_{(A,B)}=\sup\{\sqrt{1/\beta}: \mathcal{L}(Q,L,\beta) > 0 \text{ holds for some } Q \text{ and } L\}.
\end{equation}
To prove this claim, observe from Lemma~\ref{lem:mod_slem} and \eqref{eq:def_lam_mod} that $\lambda^\textup{a}_{(A,B)}$ is equal to the supremum value of $\rho$ such that \eqref{eq:addp_mod_opt_com1_s} holds for some $P$, $K$, and $\gamma > 0$. We introduce the new variable $\bar P=P/\gamma$. We also note that $\rho$ can always be taken to be positive, thus, we suppose that $\rho > 0$. Using a Schur complement argument for the matrix in \eqref{eq:addp_mod_opt_com1_s}, we have
\begin{equation}
\label{eq:addp_mod_opt_schag1_s}
\begin{bmatrix}
    \bar P & 0 & A_K^\top & I  \\
    0 &  I & B^\top & 0 \\
    A_K & B & \bar P^{-1} & 0\\
    I & 0 & 0 & \frac{1}{\rho^2}I 
\end{bmatrix} > 0.
\end{equation} 
We multiply \eqref{eq:addp_mod_opt_schag1_s} from left and right by a block diagonal matrix with the diagonal blocks $\bar P^{-1}$ and $I$ to have
\begin{equation}
\label{eq:addp_mod_opt_schag1_s_mul}
\begin{bmatrix}
    \bar P^{-1} & 0 & \bar P^{-1}A_K^\top & \bar P^{-1}  \\
    0 &  I & B^\top & 0 \\
    A_K\bar P^{-1} & B & \bar P^{-1} & 0\\
    \bar P^{-1} & 0 & 0 & \frac{1}{\rho^2}I 
\end{bmatrix} > 0.
\end{equation} 
By a change of variables $Q=\bar P^{-1}$, $L=KQ$, and $\beta=1/\rho^2$, we have
\begin{equation}
\label{eq:addp_mod_opt_LMI1_s}
    \begin{bmatrix}
        Q & 0 &   QA^\top + L^\top B^\top & Q \\
        0 & I  & B^\top & 0 \\
        AQ+BL& B  & Q & 0 \\
        Q & 0 & 0 & \beta I
    \end{bmatrix} > 0.
\end{equation}
We take the Schur complement of the matrix in \eqref{eq:addp_mod_opt_LMI1_s} with respect to its second diagonal block to see that \eqref{eq:claim_mod} holds. Now, we use \eqref{eq:claim_mod} to show that $\lambda_{(A,B)}^{\text{a}}=\sqrt{1/\beta_*}$. To this end, let $Q_*$ and $L_*$ be the optimal solution of \eqref{eq:th:frag-model-opt_s} with the optimal value $\beta_*$. We claim that for any $\beta >\beta_*$, there exist $\hat Q$ and $\hat L$ such that $\mathcal{L}(\hat Q,\hat L,\beta)>0$. To show this, note that since $(A,B)$ is stabilizable, there exist $\bar Q$, $\bar L$ and $r$ such that 
\begin{equation}
\label{eq:strict_LQ}
\bar{\mathcal{L}}(\bar Q,\bar L,r)\coloneqq \begin{bmatrix}
\bar Q & \bar Q A^\top+\bar L^\top B^\top  & \bar Q \\
A\bar Q+B\bar L & \bar Q  & 0 \\
\bar Q & 0  & r I
\end{bmatrix} >0.
\end{equation}
Let $\beta>\beta_*$ and $\epsilon=r/(\beta-\beta_*)$. We take $\hat Q=Q_*+\bar Q/\epsilon$ and $\hat L=L_*+\bar L/\epsilon$. It follows from $\mathcal{L}(Q_*,L_*,\beta_*) \geq 0$ that \mbox{$\mathcal{L}(\hat Q,\hat L,\beta)=\mathcal{L}(Q_*,L_*,\beta_*)+\frac{1}{\epsilon}  \bar{\mathcal{L}}(\bar Q,\bar L,r) >0$}. This proves the claim, and, in turn, it shows that $\lambda_{(A,B)}^{\text{a}}=\sqrt{1/\beta_*}$.

What remains to be proven is that $A+B(K_*+\Delta)$ with \mbox{$K_*=L_*Q_*^{\dagger}$} is Schur for all $\| \Delta \| < \lambda_{(A,B)}^{\text{a}}$. To show this, let $\Delta \in \R^{m \times n}$ satisfy $\|\Delta\|<\sqrt{1/\beta_*}$. Since the matrix in \eqref{eq:th:frag-model-opt_s(b)} is positive semi-definite, we have $\ker Q_* \subseteq \ker(AQ_*+BL_*)$, which implies that $\ker Q_* \subseteq \ker BL_*$. By the definition of the Moore-Penrose pseudoinverse, we have \mbox{$Q_*=Q_*Q_*^{\dagger}Q_*$}. This implies that $\im (I-Q_*^{\dagger}Q_*) \subseteq \ker Q_* \subseteq \ker BL_*$, and thus, \mbox{$BL_*Q_*^{\dagger}Q_*=BL_*$}. Hence, we have $BL_*=BK_*Q_*$. By substituting $BL_*=BK_*Q_*$ into \eqref{eq:th:frag-model-opt_s(b)}, we have $\mathcal{L}(Q_*,K_*Q_*,\beta_*)\geq 0$. By taking the Schur complement of this with respect to its first diagonal block, we have
\begin{equation}
\label{eq:K*+Delta}
    \begin{bmatrix}
        Q_*\!-\!(A\!+\!BK_*) Q_* (A\!+\!BK_*)^{\!\top}\!-\!BB^\top & -(A\!+\!BK_*) Q_* \\
        -Q_* (A\!+\!BK_*)^\top & \beta_* I- Q_*
    \end{bmatrix} \!>0.
\end{equation} We multiply \eqref{eq:K*+Delta} from left and right, respectively, by $\begin{bmatrix}
    I & B\Delta
\end{bmatrix}$ and $\begin{bmatrix}
    I & B\Delta
\end{bmatrix}^\top$ to have
\begin{equation}
Q_*-(A+BK_*+B\Delta)Q_*(A+BK_*+B\Delta)^{\!\top} \!\!\geq\!\! B\!\left(\!I\!-\!\beta_*\Delta \Delta^{\!\top}\!\right)\!B^\top\!.
\end{equation}
Since $\|\Delta\|<\sqrt{1/\beta_*}$, we have $I-\beta_*\Delta \Delta^{\top}>0$. Let $\sigma>0$ be such that $I-\beta_*\Delta \Delta^{\top}\geq\sigma I$. Therefore, we have
\begin{equation}
\label{eq:lyp_like_ineq_BB}
Q_*-(A+BK_*+B\Delta)Q_*(A+BK_*+B\Delta)^{\top} \geq \sigma BB^\top.
\end{equation}
Let $v\in \mathbb{C}^n$ and $\eta \in \mathbb{C}$ be such that $(A+BK_*+B\Delta)^\top v=\eta v$. We multiply \eqref{eq:lyp_like_ineq_BB} from left and right, respectively, by $v^*$ and $v$ to have
\begin{equation}
\label{eq:1-lambda^2}
(1-|\eta|^2)v^*Q_*v \geq \sigma v^*BB^\top v.
\end{equation}
In case $B^\top v=0$, we have $(A+BK_*+B\Delta)^\top v =A^\top v=\eta v$. Since $(A,B)$ is stabilizable, it follows from the Hautus test that $|\eta|<1$. In case $B^\top v\neq0$, it follows from \eqref{eq:1-lambda^2} that $|\eta|<1$. Therefore, $A+BK_*+B\Delta$ is Schur, which completes the proof. \hfill \QED

Theorem~\ref{th:frag-model-opt} can be used to compute the optimal feedback gain $K_*$ that provides the least amount of fragility in the sense of the measure \eqref{eq:rhoK}. Nevertheless, if the feedback gain $K$ is given, one can also compute $\lambda_{(A,B)}^\textup{a}(K)$ by solving an SDP discussed in the following remark. 

\begin{remark}
\label{th:frag_model}
Suppose that $(A,B)$ is stabilizable and $B\neq 0$. Let $K$ be such that $A_K$ is Schur. Then, $\lambda_{(A,B)}^\textup{a}(K)= \sqrt{1/\beta_*}$ where $\beta_*$ is the optimal value of the following SDP:
\begin{equation}
\label{eq:th:frag-model_s}
%\begin{aligned}
\min_{Q}\ \beta  \hspace{0.15 cm} \text{s.t.}\hspace{0.15 cm}
\begin{bmatrix}
Q & QA^\top+QK^\top B^\top  & Q \\
AQ+BKQ & Q-BB^\top  & 0 \\
Q & 0  & \beta I
\end{bmatrix}\geq 0.
%\end{aligned}
\end{equation}

\end{remark}

The following example illustrates the results of Theorems~\ref{th:frag-model-opt} and Remark~\ref{th:frag_model}.

%\begin{remark}
%The control gain obtained by solving the SDP in Theorem \ref{th:frag_model} stabilizes all systems of the form $A+BK+B\Delta$ for all $\Delta\in\inter\mathcal{B}(0,\lambda_{(A,B)})$. Hence, the application of such results is not limited to stability guarantee against errors in the feedback gain, but against any disturbance of matched type \cite{}. Moreover, since the closed-loop stability is guaranteed for a common Lyapunov function, $\Delta$ can also be time-varying, that is, $A+BK+B\Delta(t)$ will be uniformly asymptotically stable if $\Delta(t)\in\inter\mathcal{B}(0,\lambda_{(A,B)})$. 
%\end{remark}

\begin{example}
\label{ex:2}
Consider $A=\begin{bmatrix}
        1 & 1 \\ 0 & 1
    \end{bmatrix}$ and $B=\begin{bmatrix}
        0.5 \\ 1
    \end{bmatrix}$.
The set of all stabilizing feedback gains is shown by the triangle area in Fig.~\ref{fig:2a}. In particular, for $K=-\begin{bmatrix}
1 & 1
\end{bmatrix}$ we have $\mu_{(A,B)}^{\text{a}}(K)=0.447$, and Remark~\ref{th:frag_model} yields $\lambda_{(A,B)}^{\text{a}}(K)=0.333$. According to Theorem~\ref{th:frag-model-opt}, we have $\lambda_{(A,B)}^{\text{a}}=0.667$ that is attained by $K_*=-\begin{bmatrix}
0.667 & 1.333
\end{bmatrix}$. Fig.~\ref{fig:2b} provides a contour plot illustrating the level sets of stabilizing feedback gains with constant $\lambda^\textup{a}_{(A,B)}(K)$.
\end{example}

\begin{figure}[h!]
\centering
    \includegraphics[width=0.9\columnwidth]{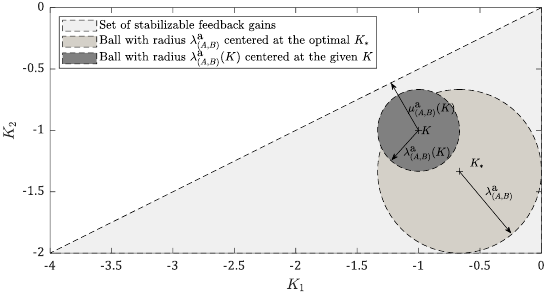}
    \caption{A visualization of the values of $\mu_{(A,B)}^\textup{a}(K)$ and $\lambda_{(A,B)}^\textup{a}(K)$ for a certain $K$, and the value of $\lambda_{(A,B)}^\textup{a}$ with the optimal $K_*$ for Example~\ref{ex:2}.}
    \label{fig:2a} 
\end{figure}
\begin{figure}[h!]
\centering
    \includegraphics[width=0.9\columnwidth]{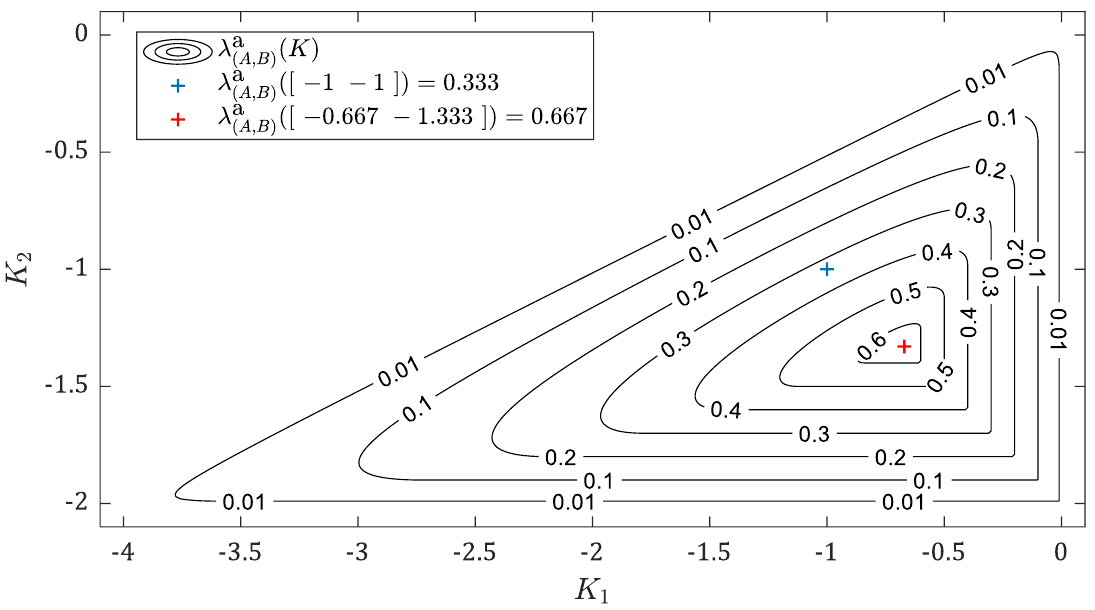}    
    \caption{Contours of constant $\lambda_{(A,B)}^\textup{a}(K)$ over all stabilizing feedback gains for Example~\ref{ex:2}.}
    \label{fig:2b}  
\end{figure}

\subsection{Data-driven analysis}

In this section, we turn our attention to the data-driven setting. Assume that the data $\mathcal{D}$ are informative for stabilization. Let $K$ be such that $A+BK$ is Schur for all $(A,B)\in\Sigma_\mathcal{D}$. Analogous to \eqref{eq:psiK}, we define
\begin{equation}
\begin{split}
\mu_{\mathcal{D}}^{\text{a}}(K)\coloneqq \sup \{\rho:A+BK+B\Delta\ \text{is Schur for all}\\
\Delta\in\mathcal{B}(0,\rho)\ \text{and all}\ (A,B)\in\Sigma_\mathcal{D}\}.
\end{split}
\end{equation}
The value of $\mu_{\mathcal{D}}^{\text{a}}(K)$ gives the largest bound for an additive feedback perturbation so that the closed-loop remains stable for all systems within $\Sigma_\mathcal{D}$. The two extreme cases where $\mu_{\mathcal{D}}^{\text{a}}(K)$ is zero or not finite are fully characterized by the following theorem.

\begin{theorem}
\label{prop:mu^a_bounded} 
Suppose that the data $\mathcal{D}$ are informative for stabilization. Let $K$ be such that $A+BK$ is Schur for all $(A,B)\in\Sigma_\mathcal{D}$. Then, the following statements hold:
\begin{enumerate}[label=(\alph*),ref=\ref{prop:mu^a_bounded}(\alph*)]
    \item\label{prop:mu^a_bounded(a)} $\mu_{\mathcal{D}}^{\text{a}}(K)=0$ \emph{if and only if} $\rank\begin{bmatrix}X_-^\top & U_-^\top\end{bmatrix}<n+m$.
    \item\label{prop:mu^a_bounded(b)} $\mu_{\mathcal{D}}^{\text{a}}(K)$ is not finite \emph{if and only if} $\Sigma_\mathcal{D}=\{(A_\text{true},B_\text{true})\}$ such that $B_\text{true}= 0$.
\end{enumerate}
\end{theorem}
\begin{proof}
(a) To prove the ``if'' part, suppose that $\rank\begin{bmatrix}X_-^\top & U_-^\top\end{bmatrix}<n+m$. Let nonzero $(A_0,B_0)$ be such that $A_0X_-+B_0U_-=0$. Also, let $\rho=\mu_\mathcal{D}^\text{a}(K)$. By the definition, every $\tilde{K}\in\mathcal{B}(K,\rho)$ has the property that $A+B\tilde{K}$ is Schur for all $(A,B)\in\Sigma_\mathcal{D}$. Based on \cite[eq. (27)]{HenkSlemma} (also see \cite[Lem. 15]{van2020data}), every such $\tilde{K}$ satisfies $A_0+B_0 \tilde{K}=0$. Thus, we have $A_0+B_0 K=0$. Since the data are informative for stabilization, we have $\rank X_-=n$ (see \mbox{\cite[Thm. 16]{van2020data}}). This implies that $B_0\neq 0$. Take $\Delta=(\rho/\|B_0\|) B_0^\top$ and observe that $A_0+B_0 K+B_0\Delta=0$ implies that $\rho B_0B_0^\top=0$. Since $B_0\neq 0$, we have $\rho=0$. We prove the ``only if'' part by contraposition. Suppose that $\begin{bmatrix}X_-^\top & U_-^\top\end{bmatrix}$ has full column rank. This implies that $\Sigma_\mathcal{D}$ is bounded. In this case, there exist $K$ and $\beta>0$ such that for every $(A,B)\in\Sigma_\mathcal{D}$ the Lyapunov inequality $P-(A+BK)P(A+BK)^\top \geq \beta I$ holds for some $P>0$. Since $\beta>0$, there exists a small enough $\rho>0$ such that for every $(A,B)\in\Sigma_\mathcal{D}$, the Lyapunov inequality $P-(A+B\tilde{K})P(A+B\tilde{K})^\top> 0$ holds for all $\tilde{K}\in\mathcal{B}(K,\rho)$. This implies that $\mu_{\mathcal{D}}^{\text{a}}(K)>0$. 

(b) The ``if'' part is obvious. For the ``only if'' part, suppose that $\mu_{\mathcal{D}}^{\text{a}}(K)$ is not finite. In view of part (a), this implies that $\rank \begin{bmatrix}X_-^\top & U_-^\top\end{bmatrix}=n+m$. Let $(A,B)\in\Sigma_\mathcal{D}$. Since $\mu_{\mathcal{D}}^{\text{a}}(K)$ is not finite, we have that $\mu_{(A,B)}^{\text{a}}(K)$ is also not finite. Due to Theorem~\ref{prop:Bnonzero}, this implies that $B=0$. This shows that every $(A,B)\in\Sigma_\mathcal{D}$ satisfies $B=B_\text{true}=0$. What remains to be proven is that $\Sigma_\mathcal{D}$ is a singleton. To this end, define $L=(N|\bar{N}_{22})^{\frac{1}{2}}$, $R=(-\bar{N}_{22})^{-\frac{1}{2}}$, and $\begin{bmatrix}
\hat{A} & \hat{B}
\end{bmatrix}=-\bar{N}_{12}\bar{N}_{22}^{-1}$.
%\begin{equation}
%L=(N|\bar{N}_{22})^{\frac{1}{2}},
%\end{equation}
%\begin{equation}
%R=\begin{bmatrix}
%    -N_{22} & -N_{23} \\
%    -N_{32} & -N_{33}
%    \end{bmatrix}^{-\frac{1}{2}}, \ \text{and}
%\end{equation}
%\begin{equation}
%\begin{bmatrix}
%\hat{A} & \hat{B}
%\end{bmatrix}=-\begin{bmatrix}
%    N_{21} \\ N_{31}
%    \end{bmatrix}^\top\begin{bmatrix}
%    N_{22} & N_{23} \\
%    N_{32} & N_{33}
%    \end{bmatrix}^{-1}.
%\end{equation}
According to the parametrization provided in \cite[Thm. 3.3]{HenkQMI2023}, one can verify that we have 
\begin{equation}
\label{eq:Exp_parameter}
\Sigma_\mathcal{D}=\{(A,B):\begin{bmatrix}
A & B
\end{bmatrix}=\begin{bmatrix}
\hat{A} & \hat{B}
\end{bmatrix}+LSR,SS^\top \leq I\}.
\end{equation}
Let $R_A\in\mathbb{R}^{(n+m)\times n}$ and $R_B\in\mathbb{R}^{(n+m)\times m}$ be such that $\begin{bmatrix}
R_A & R_B
\end{bmatrix}=R$. Note that $R>0$, thus, $R_B\neq0$. Now, since every data-consistent system $(A,B)$ satisfies $B=0$, we have $LSR_B=0$ for all $SS^\top \leq I$. This, together with $R_B\neq0$, implies that $L=0$. Therefore, the set $\Sigma_\mathcal{D}$ is a singleton due to \eqref{eq:Exp_parameter}. 
\end{proof}

\begin{remark}
\label{rem:1}
In case $\mu_{\mathcal{D}}^{\text{a}}(K)=0$, a small additive perturbation on the feedback gain, no matter how small, destabilizes a nonempty subset of $\Sigma_\mathcal{D}$. Therefore, a small additive perturbation may also destabilize the true system. We refer to this case as \emph{extreme fragility}. Theorem~\ref{prop:mu^a_bounded(a)} shows that if $\begin{bmatrix}
X_-^\top & U_-^\top
\end{bmatrix}$ does not have full column rank, then any data-driven stabilizing feedback gain is extremely fragile. This condition is equivalent to the case where the set of data-consistent systems is unbounded. Therefore, in the noise-free case, the data-driven feedback gain is extremely fragile \emph{if and only if} the system cannot be uniquely identified, see \cite[Prop. 6]{van2020data}. In case $\mu_{\mathcal{D}}^{\text{a}}(K)$ is not finite, the data-driven feedback gain $K$ is \emph{immune} to additive perturbations. Theorem~\ref{prop:mu^a_bounded(b)} reveals that this can only happen in the rare case where $B_{\textup{true}}=0$ and the system can be uniquely identified from noisy data. This motivates the study of fragility in a data-driven setting as it shows that such feedback gains typically have a limited tolerance towards additive perturbations.
\end{remark}

Now, we aim at developing numerically tractable methods to approximate $\mu_{\mathcal{D}}^{\text{a}}(K)$. Suppose that the data $\mathcal{D}$ are informative for quadratic stabilization. Let $K$, $P>0$, and $\alpha\geq 0$ be such that $K\in\mathcal{K}(P,\alpha)$. We define
\begin{equation}
\kappa_{\mathcal{D}}^{\text{a}}(P,\alpha,K)\coloneqq \sup \{\rho:K+\Delta \in \mathcal{K}(P,\alpha) \ \forall\Delta \in \mathcal{B}(0,\rho)\}.
%\kappa_{\mathcal{D}}^{\text{a}}(P,\alpha,K)\coloneqq \sup \{\rho:\mathcal{B}(K,\rho)\subseteq\mathcal{K}(P,\alpha)\}.
\end{equation}
One can also define the largest value of $\kappa_{\mathcal{D}}^{\text{a}}(P,\alpha,K)$ over all $P > 0$ and $\alpha \geq 0$ as
\begin{equation}
\label{eq:def_lam_K}
    \lambda^\text{a}_\mac{D}(K) \coloneqq\sup\{\kappa^a_\mac{D}(P,\alpha,K):P >0,\alpha \geq 0\}.
\end{equation} 
Now, we define the largest value $\lambda^\text{a}_\mac{D}(K)$ over all $K \in \mathcal{K}$ as 
\begin{equation}
\label{eq:def_lam_optK}
    \lambda^\text{a}_\mac{D}\coloneqq \sup \{\lambda^\text{a}_\mac{D}(K):K \in \mac{K}\}.
\end{equation}
The value of $\lambda^\text{a}_\mac{D}$ can be obtained by solving an SDP as stated next.

\begin{theorem}
\label{thm:add_lambda}
Suppose that $\rank\begin{bmatrix}
        X^\top_- & U^\top_-
    \end{bmatrix}=n+m$, the data $\mac{D}$ are informative for quadratic stabilization, and $N|\bar{N}_{22}\neq0$. Then, $\lambda^\text{a}_\mac{D}=\sqrt{\beta_*}$ where $\beta_*$ is the optimal value of the following SDP:
    \begin{subequations}
    \mathtoolsset{showonlyrefs=false}
    \label{eq:SDP_add_lambda_s}
    \begin{align}
    \label{eq:SDP_add_lambda_s(a)}
        \max_{Q,L,\zeta} \ \beta & \\ 
    \label{eq:SDP_add_lambda_s(b)}
        \text{s.t.} & \ \begin{bmatrix}
            Q & 0 & 0 & 0 & 0\\
            0 & -Q & -L^\top & -Q & 0\\
            0 & -L & -\beta I & 0 & L\\
            0 & -Q & 0 & I & Q\\
            0 & 0 & L^\top & Q & Q 
        \end{bmatrix}\!-\zeta\begin{bmatrix}
            N & 0 \\ 0 & 0
        \end{bmatrix} \geq 0.
    \end{align}
    \end{subequations} Moreover, provided $L_*$ and $Q_*$ are optimizers of \eqref{eq:SDP_add_lambda_s}, \mbox{$A+B(K_*+\Delta)$} with $K_*=L_*Q_*^\dagger$ is Schur for all $\|\Delta\| < \lambda^\text{a}_\mac{D}$ and all $(A,B) \in \Sigma_\mac{D}$.
\end{theorem}

To prove this theorem, we need an auxiliary result presented in the following Lemma.

\begin{lemma}
\label{lem:dd_slem}
Suppose that $\rank\begin{bmatrix}
        X^\top_- & U^\top_-
    \end{bmatrix}=n+m$, the data $\mac{D}$ are informative for quadratic stabilization. Let $P>0$, $K$ and $\alpha \geq 0$ be such that $K \in \mathcal{K}(P,\alpha)$. Then, we have $\rho<\kappa_{\mathcal{D}}^{\text{a}}(P,\alpha,K)$ \emph{if and only if} there exists $\gamma > 0$ such that the following LMI is satisfied:
     \begin{equation}
    \label{eq:SDP_add_lambda_s_s_K}
        \begin{bmatrix}
            P & 0 & 0 & 0 & 0\\
            0 & -P & -PK^\top & -P & 0\\
            0 & -KP & -\gamma\rho^2 I & 0 & KP\\
            0 & -P & 0 & \gamma I & P\\
            0 & 0 & PK^\top & P & P 
        \end{bmatrix}-\alpha\begin{bmatrix}
            N & 0 \\ 0 & 0
        \end{bmatrix} > 0.
\end{equation}
\end{lemma}
\begin{proof}
We define $\bar M_\rho \coloneqq \begin{bmatrix}
        \rho^2I_m & 0 \\ 0 & -I_n
    \end{bmatrix}$ and
\begin{equation}
\label{eq:def_MDelta}
\begin{split}
    M_{K}\coloneqq-\begin{bmatrix}
        \alpha N_{13} & 0 \\ PK^\top+\alpha N_{23} & P 
    \end{bmatrix}^\top \Theta^{-1}\begin{bmatrix}
        \alpha N_{13} & 0 \\ PK^\top+\alpha N_{23} & P 
    \end{bmatrix} \\
    -\begin{bmatrix}
        K \\ I
    \end{bmatrix}P\begin{bmatrix}
        K \\ I
    \end{bmatrix}^\top - \alpha \begin{bmatrix}
        N_{33} & 0 \\ 0 & 0
    \end{bmatrix}.
\end{split}
\end{equation}  
We observe from Theorem~\ref{th:1} that $K+\Delta \in \mathcal{K}(P,\alpha)$ if and only if $K^\top+\Delta^\top \in \mac{Z}^+_n(M)$, i.e.,
\begin{equation}
    \label{eq:schurC_KPDelta}
    \begin{split}
        &-(K+\Delta)P(K+\Delta)^\top-\alpha N_{33} \\
        &-\begin{bmatrix}
        \alpha N_{13} \\ P(K+\Delta)^\top+\alpha N_{23}
    \end{bmatrix}^\top \!\!\! \Theta^{-1} \!\! \begin{bmatrix}
        \alpha N_{13} \\ P(K+\Delta)^\top+\alpha N_{23}
    \end{bmatrix} > 0.
    \end{split}
\end{equation}
Based on this inequality, one can verify that \mbox{$K^\top+\Delta^\top \in \mac{Z}^+_n(M)$} is equivalent to $\Delta^\top \in \mac{Z}^+_n(M_K)$. Observe that \mbox{$\mac{B}(0,\rho)=\mac{Z}_n(\bar M_\rho)$}. Also, observe that $\bar M_\rho \in \pmb{\Pi}_{m,n}$ and the second diagonal block of $\bar M_\rho$ is negative definite. Therefore, it follows from the matrix \mbox{S-lemma} \cite[Thm. 4.10]{HenkQMI2023} that $\mathcal{Z}_{n}(\bar M_\rho)\subseteq\mathcal{Z}_{n}^+(M_K)$ if and only if  \mbox{$M_K-\gamma \bar M_\rho> 0$} for some $\gamma\geq 0$. Since $M_K$ is not positive definite, we see that $\gamma\neq0$. Using a Schur complement argument, one can verify that $M_K-\gamma \bar M_\rho>0$ is equivalent to \eqref{eq:SDP_add_lambda_s_s_K}. 
\end{proof}

Now, we use this lemma to prove Theorem~\ref{thm:add_lambda}.

\textit{Proof of Theorem~\ref{thm:add_lambda}:}
 We first show that the LMI \eqref{eq:SDP_add_lambda_s(b)} is feasible. For this, it follows from Theorem~\ref{th:1} that there exist $\bar P>0$, $\bar L$, $\bar \alpha \geq 0$ and $\nu>0$ such that
\begin{equation}
\bbma \bar P & 0 & 0 & 0 \\ 0 & -\bar P & - \bar L^\top &  0 \\ 0 & -\bar L & -\nu I &  \bar L \\ 0 & 0 & \bar L^\top & \bar P \ebma
-\bar \alpha \bbma N & 0\\ 0 & 0 \ebma > 0.
\end{equation}
Thus, we observe that there exists a sufficiently large $\bar \gamma > 0$ such that \eqref{eq:SDP_add_lambda_s(b)} holds with $Q=\bar P/\bar \gamma$, $L=\bar L/\bar \gamma$ and $\beta=\nu/\bar \gamma$. Now, we show that the SDP \eqref{eq:SDP_add_lambda_s} is attained. We observe from the last $2n \times 2n$ diagonal block of the matrix in \eqref{eq:SDP_add_lambda_s(b)} that every feasible $Q$ satisfies $0\leq Q\leq I$. This, in turn, implies that there exists a $\delta>0$ such that every feasible $L$ satisfies $\|L\|\leq \delta$. Since $N|\bar{N}_{22}\neq0$, $N$ has at least one positive eigenvalue. This implies that there exists a $\bar \zeta \geq 0$ such that every feasible $\zeta$ satisfies $0\leq \zeta \leq \bar \zeta$. Now, it is evident from the third diagonal block that there exists a $\bar\beta \geq 0$ such that every feasible $\beta$ satisfies $\beta \leq \bar\beta$. Hence, the set of all feasible solutions is bounded. Therefore, the SDP \eqref{eq:SDP_add_lambda_s} is attained.

Next, we prove that $\lambda_{\mathcal{D}}^\textup{a}=\sqrt{\beta_*}$. For this, define
 \begin{equation}
    \label{eq:SDP_add_lambda_s_s}
        \mathcal{L}_{\mathcal{D}}(Q,L,\zeta,\beta)\coloneqq\begin{bmatrix}
            Q & 0 & 0 & 0 & 0\\
            0 & -Q & -L^\top & -Q & 0\\
            0 & -L & -\beta I & 0 & L\\
            0 & -Q & 0 & I & Q\\
            0 & 0 & L^\top & Q & Q 
        \end{bmatrix}-\zeta\begin{bmatrix}
            N & 0 \\ 0 & 0
        \end{bmatrix}.
\end{equation}
We claim that
\begin{equation}
\label{eq:clia_dd}
    \lambda_{\mathcal{D}}^\textup{a}=\sup \{\sqrt{\beta}: \mathcal{L}_{\mathcal{D}}(Q,L,\zeta,\beta)>0 \text{ holds for some $Q$, $L$ and $\zeta$}\}.
\end{equation}
To prove this claim, we observe from Lemma~\ref{lem:dd_slem} that $\lambda_\mathcal{D}^\textup{d}$ is equal to the supremum value of $\rho$ such that \eqref{eq:SDP_add_lambda_s_s_K} holds for some $P$, $K$, $\alpha\geq 0$, and $\gamma > 0$. Take $Q=P/\gamma$, $L=KQ$, and $\zeta=\alpha/\gamma$, one can verify that $\mathcal{L}_{\mathcal{D}}(Q,L,\zeta,\beta)> 0$ is equivalent to \eqref{eq:SDP_add_lambda_s_s_K}. Thus, the claim holds. This claim can be used to show that $\lambda_{\mathcal{D}}^\textup{a}=\sqrt{\beta_*}$. To this end, we prove that there exist $\hat Q$, $\hat L$, and $\hat \zeta$ such that \mbox{$\mathcal{L}_{\mathcal{D}}(\hat Q,\hat L,\hat \zeta,\beta)>0$} for any $\beta < \beta_*$. Since the data $\mathcal{D}$ are informative for quadratic stabilization, it follows from Lemma~\ref{lem:dd_slem} that there exist $\bar Q$, $\bar L$, and $\bar \zeta$ such that $\mathcal{L}_{\mathcal{D}}(\bar Q,\bar L,\bar \zeta,0)>0$. Let $Q_*$, $L_*$ and $\zeta_*$ be optimal solutions for the SDP \eqref{eq:SDP_add_lambda_s}, i.e., $\mathcal{L}_{\mathcal{D}}(Q_*,L_*,\zeta_*,\beta_*)\geq 0$. Let \mbox{$\epsilon=\beta/\beta_*$}. We take $\hat Q=\epsilon Q_*+(1-\epsilon)\bar Q$, $\hat L=\epsilon L_*+(1-\epsilon)\bar L$ and $\hat \zeta=\epsilon \zeta_*+(1-\epsilon)\bar \zeta$. We observe that
$$
\mathcal{L}_{\mathcal{D}}(\hat Q,\hat L,\hat \zeta,\beta)=\epsilon \mathcal{L}_{\mathcal{D}}(Q_*,L_*,\zeta_*,\beta_*)+(1-\epsilon)\mathcal{L}_{\mathcal{D}}(\bar Q,\bar L,\bar \zeta,0).
$$
It follows from $\mathcal{L}_{\mathcal{D}}(Q_*,L_*,\zeta_*,\beta_*)\geq0$ that $\mathcal{L}_{\mathcal{D}}(\hat Q,\hat L,\hat \zeta,\beta)>0$. This implies that $\beta^*$ is equal to the supremum value of $\beta$ such that $\mathcal{L}_{\mathcal{D}}(Q,L,\zeta,\beta)>0$, i.e., $\lambda_{\mathcal{D}}^\textup{a}=\sqrt{\beta_*}$.

What remains to be proven is that $A+BK_*+B\Delta$ is Schur with $K_*=L_*Q_*^{\dagger}$ for all $\|\Delta\|<\lambda_{\mathcal{D}}^\textup{a}$ and all \mbox{$(A,B)\in\Sigma_\mathcal{D}$}. To show this, let $\Delta \in \R^{m \times n}$ satisfy $\|\Delta\|<\sqrt{\beta_*}$ and $(A,B)\in\Sigma_\mathcal{D}$. Observe from $\mathcal{L}_{\mathcal{D}}(Q_*,L_*,\zeta_*,\beta_*)\geq0$ that we have $\ker Q_*\subseteq\ker L_*$. We also have $Q_*Q_*^\dagger Q_*=Q_*$, which implies that \mbox{$\im (I=Q_*^\dagger Q_*)\subseteq \ker Q_*\subseteq \ker L_*$}. Hence, we have \mbox{$L_*Q_*^\dagger Q_*=L_*$}. This implies that $L_*=KQ_*$. Therefore, \mbox{$\mathcal{L}_{\mathcal{D}}(Q_*,KQ_*,\zeta_*,\beta_*)\geq 0$}.
Take the Schur complement of this with respect to the last $n\times n$ diagonal block to have
\begin{equation}
\label{eq:proof_some_ineq}
\begin{bmatrix}
Q_* & 0 & 0 & 0 \\
0 & -Q_* & -Q_*^\top K_*^\top & -Q_* \\
0 & -K_* Q_* & -K_* Q_*K_*^\top \!-\!\beta_* I & -K_*Q_* \\
0 & -Q_* & -Q_* K_*^\top & I-Q_*
\end{bmatrix}\!-\zeta_*\! \begin{bmatrix}
N & 0 \\
0 & 0
\end{bmatrix}\!\geq\! 0.
\end{equation}
It follows from \eqref{eq:explaining_QMI_induced} that $\begin{bmatrix}
I \!\!&\!\! A \!\!&\!\!  B
\end{bmatrix} N\begin{bmatrix}
I \!\!&\!\! A \!\!&\!\!  B
\end{bmatrix}^{\!\top}\!\geq\! 0$. Using this, we multiply \eqref{eq:proof_some_ineq} from left and right, respectively, by $\begin{bmatrix}
I & A & B & B\Delta
\end{bmatrix}$ and its transpose to have
\begin{equation}
Q_*-(A+BK_*+B\Delta)Q_*(A+BK_*+B\Delta)^\top\geq B(\beta_* I-\Delta \Delta ^\top)B^\top.
\end{equation}
Since the data are informative for quadratic stabilization, the pair $(A,B)$ is stabilizable. Now, using the same argument as in the proof of Theorem~\ref{th:frag-model-opt}, one can verify that since $\|\Delta\|<\sqrt{\beta_*}$, we have that $A+BK_*+B\Delta$ is Schur. This completes the proof. \hfill \QED

Theorem~\ref{thm:add_lambda} is the data-driven counterpart of Theorem~\ref{th:frag-model-opt}, which provides the least fragile data-driven feedback gain in the sense of measure \eqref{eq:def_lam_K}. To analyze the fragility of a given feedback gain in the data-driven setting, one can solve the SDP provided by the following remark. 

\begin{remark}
\label{reM:lam_dd_optK}
Suppose that $\rank\begin{bmatrix}
        X^\top_- & U^\top_-
    \end{bmatrix}=n+m$, the data $\mac{D}$ are informative for quadratic stabilization, and $N|\bar{N}_{22}\neq0$. Let $K \in \mac{K}(P,\alpha)$ for some $P>0$ and $\alpha \geq 0$. Then, $\lambda^\text{a}_\mac{D}(K)=\sqrt{\beta_*}$ where $\beta_*$ is the optimal value of the following SDP:
    \begin{equation}
    \label{eq:SDP_add_lambda_s_K}
    %\begin{aligned}
        \max_{Q,\zeta} \ \beta \ \text{s.t.}  
        \begin{bmatrix}
            Q & 0 & 0 & 0 & 0\\
            0 & -Q & -QK^\top & -Q & 0\\
            0 & -KQ & -\beta I & 0 & KQ\\
            0 & -Q & 0 & I & Q\\
            0 & 0 & QK^\top & Q & Q 
        \end{bmatrix}\!-\zeta\begin{bmatrix}
            N & 0 \\ 0 & 0
        \end{bmatrix} \geq 0.
    %\end{aligned}
    \end{equation}
\end{remark}

%Since $\mac{K}(P,\alpha)$ is closed and bounded (see Theorem \ref{th:1}), $\kappa_{\mathcal{D}}^{\text{a}}(P,\alpha,K)$ is well-defined.
%It is evident that $\kappa_{\mathcal{D}}^{\text{a}}(P,K)\leq\lambda_{\mathcal{D}}^{\text{a}}(K)\leq\mu_{\mathcal{D}}^{\text{a}}(K)$. The value of $\lambda^\text{a}_\mac{D}(K)$ can be obtained by solving an SDP as stated next.

The following example illustrates the results of Theorems~\ref{thm:add_lambda} and Remark~\ref{reM:lam_dd_optK}.

\begin{example}
\label{ex:3}
Consider the true system \eqref{dy-lsys} with
\begin{equation}
A_\text{true}\coloneqq \begin{bmatrix}
1 & 1 \\
0 & 1
\end{bmatrix},\hspace{0.5 cm} B_\text{true}\coloneqq \begin{bmatrix}
0.5 \\
1
\end{bmatrix}.
\end{equation}
Assume that the noise sequence satisfies $\|W_-\|\leq 1$, which can be modeled by $\Phi_{11}=I$, $\Phi_{12}=\Phi_{21}^\top=0$, and $\Phi_{22}=-I$. Consider the collected input-state data and the noise signal as follows:
\begin{equation}
\label{eq:data-table}
\begin{array}{c|cccccc}
t      & 0 & 1 & 2 & 3 & 4 \\ \hline
u(t)   & 2 & -4 & 3 & 5 & - \\ \hline
x_1(t) & 0 & 1 & 2 & 1.5 & 5 \\
x_2(t) & 0 & 2 & -2 & 1 & 5 \\ \hline
w_1(t) & 0 & 1 & 0 & 0 & - \\
w_2(t) & 0 & 0 & 0 & -1 & -
\end{array}
\end{equation}
The LMIs in \eqref{eq:lmi1} are feasible\footnote{For the numerical examples of this paper, the LMIs and SDPs are solved using the YALMIP toolbox \cite{lofberg2004yalmip} of MATLAB with the MOSEK solver \cite{mosek}.}, hence, the data are informative for quadratic stabilization. The set $\mac{K}$ is shown in Fig.~\ref{fig:addp_dda}. In particular, for $K=-\begin{bmatrix}
    1.35 & 1.7
\end{bmatrix}$, according to Remark~\ref{reM:lam_dd_optK} we have \mbox{$\lambda_\mathcal{D}^\textup{a}(K)=0.055$}. Theorem~\ref{thm:add_lambda} yields $\lambda_\mathcal{D}^\textup{a}=0.087$, which is attained by $K_*=-\begin{bmatrix}
    1.426 & 1.782
\end{bmatrix}$. Fig.~\ref{fig:addp_ddb} provides a contour plot illustrating the level sets of stabilizing feedback gains with constant values of $\lambda_\mac{D}^\textup{a}(K)$.

%\begin{figure}[h!]
%    \centering
%    \includegraphics[width=0.9\columnwidth]{figures/setK_ballKopt_ballK.pdf}
%    \caption{The contour of $\lambda_\mathcal{D}^\textup{a}(\begin{bmatrix}
 %   K_1 & K_2
%\end{bmatrix})$ along $K_1$ and $K_2$.}
%    \label{fig:contour_add_dd_lamK}
%\end{figure}
\end{example}

\begin{figure}[h]
    \centering
        \includegraphics[width=0.9\columnwidth]{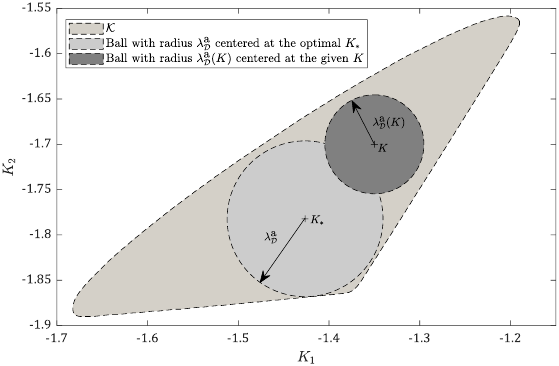}
        \caption{A visualization of the value of $\lambda_\mac{D}^\textup{a}(K)$ for a certain $K$, and the value of $\lambda_\mac{D}^\textup{a}$ with the optimal $K_*$ for Example~\ref{ex:3}.}
        \label{fig:addp_dda}
    \end{figure}
    \begin{figure}[h]
        \centering
        \includegraphics[width=0.9\columnwidth]{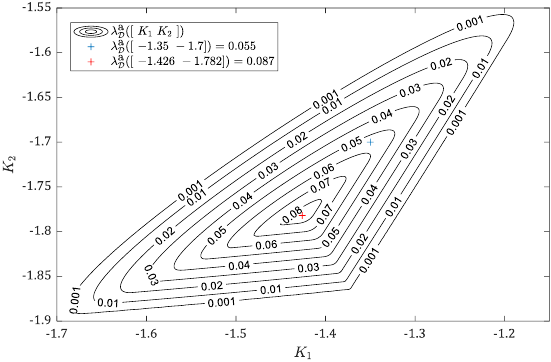}
        \caption{Contours of constant $\lambda_\mac{D}^\textup{a}(K)$ for Example~\ref{ex:3}.}
        \label{fig:addp_ddb}
\end{figure}

\begin{comment}
\begin{example}
\label{ex:3}
Consider Example \ref{ex:1}. For the computed $P$ in \eqref{eq:ex1_P} and $\alpha=1$, the set $\mathcal{K}(P,\alpha)$ is shown in Fig. \ref{fig:3} by the light gray area defined by an ellipse. For a given $K=-\begin{bmatrix}
1.35 & 1.72
\end{bmatrix}$, using Theorem \ref{th:4} we have $\kappa_{\mathcal{D}}^{\text{a}}(P,\alpha,K)=0.04$. Based on Theorem \ref{th:frag_data}, we have $\kappa_{\mathcal{D}}^{\text{a}}(P,\alpha)=0.06$ with $K_*=-\bbma 1.4171  & 1.7768 \ebma$.
\end{example}

\begin{figure}[h!]
    \centering
    \includegraphics[width=0.9\columnwidth]{figures/Doc1.pdf}
    \caption{A visualization of the notions $\kappa_{\mathcal{D}}(K)$, $\kappa_{\mathcal{D}}(P)$ for Example \ref{ex:3}.}
    \label{fig:3}
\end{figure}
\end{comment}

The following example discusses a more practical case study compared to that of Example~\ref{ex:3}, where we demonstrate the effect of the noise intensity on the fragility of the feedback gains. 

\begin{example}
\label{ex:nontoy}
Consider the state-space model of a fighter aircraft \cite[Ex. 10.1.2]{skelton1997unified} as a benchmark example\footnote{The Matlab files for this example are available at \href{https://github.com/Yongzhang-Li/Nonfragile-Data-driven-Control}{https://github.com/Yongzhang-Li/Nonfragile-Data-driven}.}. We discretize the continuous-time model with a sample time of $0.01$ to have
\begin{equation}
\begin{split}
A_\text{true}\!&=\!\begin{bmatrix}
    1.000 \!&\! -0.374 \!&\! -0.190 \!&\! -0.321 \!&\! 0.056  \!&\! -0.026 \\
    0.000 \!&\! 0.982  \!&\! 0.010  \!&\! -0.000 \!&\! -0.003 \!&\! 0.001  \\
    0.000 \!&\! 0.115  \!&\! 0.975  \!&\! -0.000 \!&\! -0.269 \!&\! 0.191  \\
    0.000 \!&\! 0.001  \!&\! 0.010  \!&\! 1.000  \!&\! -0.001 \!&\! 0.001  \\
    0.000 \!&\! 0.000  \!&\! 0.000  \!&\! 0.000  \!&\! 0.741  \!&\! 0.000  \\
    0.000 \!&\! 0.000  \!&\! 0.000  \!&\! 0.000  \!&\! 0.000  \!&\! 0.741
\end{bmatrix}\!,\\
B_\text{true}\!&=\!\begin{bmatrix}
    0.007  \!&\! 0.000 \!&\! -0.043 \!&\! 0.000 \!&\! 0.259 \!&\! 0.000 \\
    -0.003 \!&\! 0.000 \!&\! 0.030  \!&\! 0.000 \!&\! 0.000 \!&\! 0.259
\end{bmatrix}^\top.
\end{split}
\end{equation} \vspace{-0.5 cm}

To illustrate the effect of noise intensity on the value of $\lambda_{\mathcal{D}}^\textup{a}$, we perform Monte Carlo Simulations with $1000$ scenarios. For each scenario, we collect $T=30$ input and state data samples from the true system. We draw the initial state at random from a uniform distribution between $-1$ and $1$. The input and the noise signals are taken as $U_-=5U/\|U\|$ and $W_-(t)=\epsilon W/\|W\|$, where the entries of $U\in \R^{2\times T}$ and $W \in \R^{6\times T}$ are drawn at random from a uniform distribution between $-1$ and $1$. Here, parameter $\epsilon\geq 0$ represents the intensity of the noise. The noise model is considered as \eqref{eq:ass1-1} with $\Phi_{11}=\epsilon^2 I_n$, $\Phi_{12}=\Phi_{21}^\top=0$, and $\Phi_{22}=-I_T$. For each $\epsilon$, the average value and the standard deviation of $\lambda_{\mathcal{D}}^\textup{a}$ over all scenarios are denoted by $\bar \lambda_{\mathcal{D}}^\textup{a}(\epsilon)$ and $\sigma_{\mathcal{D}}^\textup{a}(\epsilon)$, respectively. Fig.~\ref{fig:lambda_change_overnoise} shows the value of $\bar \lambda_{\mathcal{D}}^\textup{a}(\epsilon)$ with respect to $\epsilon$. It can be seen that $\lambda_{\mathcal{D}}^\textup{a}(0)$ has the highest value as it corresponds to the noise-free scenario. This value is in fact equal to $\lambda_{(A_\text{true},B_\text{true})}^\textup{a}=2.976$. By increasing the noise intensity $\epsilon$, the value of $\lambda_{\mathcal{D}}^\textup{a}(\epsilon)$ decreases monotonically to zero. However, the standard deviation $\sigma_{\mathcal{D}}^\textup{a}(\epsilon)$ is equal to zero for $\epsilon=0$, it reaches its highest value for $\epsilon=0.002$, and then it decreases to zero as $\epsilon\rightarrow\infty$. Since the magnitude of the input signal remains the same, this also demonstrates the effect of a signal-to-noise ratio on the fragility of data-driven feedback gains, which is measured by $\lambda_{\mathcal{D}}^\textup{a}$.

\begin{figure}[H]
    \centering
    \includegraphics[width=0.95\linewidth]{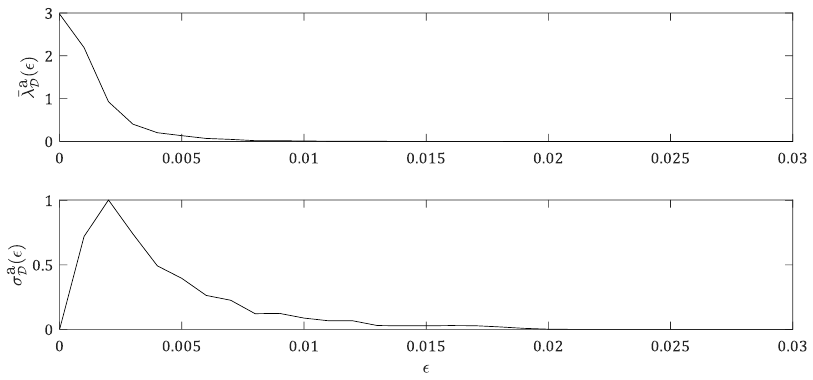}
    %\vspace{-0.5cm}
    \caption{The average value and the standard deviation of $\lambda_{\mathcal{D}}^\textup{a}$ against noise intensity $\epsilon$.}
    \label{fig:lambda_change_overnoise}
\end{figure}

\end{example}
%which satisfies $\lambda^\text{a}_\mathcal{D}(K_o)<\|\Delta\|=0.353 <\lambda_\mathcal{D}^\text{a}$. One can verify that $A_\text{true}+B_\text{true}(K_*+\Delta)$ is Schur. However, $A_\text{true}+B_\text{true}(K_o+\Delta)$ is not Schur as it has an eigenvalue equal to $1.016$.%, which means that with $K_o$ the closed-loop system does not tolerate this perturbation. 

\section{Conclusions}
\label{sec:V}

It has been shown that the fragility of a data-driven feedback gain can be quantified by means of a measure, and the least fragile data-driven feedback gain can be computed by solving a data-based SDP. In addition, it has been shown that extreme fragility and complete immunity of a data-driven feedback gain towards feedback perturbations can be fully characterized by conditions that only depend on input-state data and the noise model. In this work, we only focused on \emph{additive} perturbation on the control parameters. Another type of feedback perturbation that is relevant in practical applications of data-driven controllers is the \emph{multiplicative} one, which can capture the effect of various faults and failures such as actuator/sensor misalignment. The analysis of multiplicative perturbations is particularly challenging since the magnitude of the feedback gain directly influences the perturbation intensity. These difficulties are further amplified when additive and multiplicative perturbations occur simultaneously. The study of this and several types of structured feedback perturbations is left as future work.

\section*{References}
\bibliographystyle{IEEEtran}
\bibliography{biblo}

\end{document}